\documentclass[12pt, reqno]{amsart}
\usepackage{amsmath, amsthm, amscd, amsfonts, amssymb, graphicx, color}
\usepackage[bookmarksnumbered, colorlinks, plainpages]{hyperref}
\hypersetup{colorlinks=true,linkcolor=red, anchorcolor=green, citecolor=cyan, urlcolor=red, filecolor=magenta, pdftoolbar=true}

\newtheorem{theorem}{Theorem}[section]

\newtheorem{proposition}[theorem]{Proposition}
\newtheorem{corollary}[theorem]{Corollary}
\theoremstyle{definition}
\newtheorem{definition}[theorem]{Definition}

\newtheoremstyle{remark}
    {} 
    {0.6em} 
    {}          
    {}          
    {\bfseries} 
    {.}         
    {.5em}      
    {}          

\theoremstyle{remark}
\newtheorem{remark}{\emph{\textbf{Remark}}}[section]

\newtheorem{comment}[theorem]{Comment}
\numberwithin{equation}{section}

\begin{document}

\title{Extension operators and nonlinear structure of Banach spaces}

\author[M.A. Sofi]{M.A. Sofi$^1$}

\address{$^{1}$ Department of Mathematics, University of Kashmir, Srinagar-190006, India.}
\email{\textcolor[rgb]{0.00,0.00,0.84}{aminsofi@gmail.com}}


\subjclass[2010]{Primary 46B03, 46B20; Secondary 46C15, 28B20.}

\keywords{Banach space, Hilbert space, Lipschitz map, Extension.}

\begin{abstract}
The problem involving the extension of functions from a certain class and defined on subdomains of the ambient space to the whole space is an old and a well investigated theme in analysis. A related question whether the extensions that result in the process may be chosen in a linear or a continuous manner between appropriate spaces of functions turns out to be highly nontrivial. That this holds for the class of continuous functions defined on metric spaces is the well-known Borsuk-Dugundji theorem which asserts that given a metric space M and a subspace S of M, each continuous function g on S can be extended to a continuous function f on X such that the resulting assignment from C(S) to C(M) is a norm-one continuous linear `extension' operator.

The present paper is devoted to an investigation of this problem in the context of extendability of Lipschitz functions from closed subspaces of a given Banach space to the whole space such that the choice of the extended function gives rise to a bounded linear (extension) operator between appropriate spaces of Lipschitz functions. It is shown that the indicated property holds precisely when the underlying space is  isomorphic to Hilbert space. Among certain useful\ consequences of this theorem, we provide an isomorphic analogue of a well-known theorem of S. Reich by showing that closed convex subsets of a Banach space X arise as Lipschitz retracts of X precisely when X is isomorphically a Hilbert space. We shall also discuss the issue of bounded linear extension operators between spaces of Lipschitz functions now defined on arbitrary subsets of Banach spaces and provide a direct proof of the non-existence of such an extension operator by using methods which are more accessible than those initially employed by the authors.
\end{abstract}

 \maketitle

\section{Introduction}
The Hahn-Banach extension property for subspaces $Y$ of a Hilbert space $X$ yields the existence of a bounded linear map $\psi: Y^{*}\to X^{*}$ such that $\psi(g)|_{Y}= g$ for each $g\in Y^{*}$. (For each $g\in Y^{*}$, simply choose $f$ to be the composite of $g$ with the orthogonal projection of $X$ onto $Y$).
 Further, the existence of such a map $\psi$ for each choice of $Y$ characterises the Banach space $X$ as a Hilbert space (see Theorem $1.1$ below). A nonlinear analogue of this property could be one of the following types:
\begin{enumerate}
\item[(a)] Polynomial analogue: Existence of $\psi$ from the space of polynomials on $Y$ to the space of polynomials on $X$.
\item[(b)] Lipschitz analogue (for subspaces): Existence of $\psi$ from the space of Lipschitz functions on $Y$ to the space of Lipschitz functions on $X$.
\item[(c)] Lipschitz analogue (for subsets): Existence of $\psi$ from the space of Lipschitz functions on subsets of $X$ to the space of Lipschitz functions on $X$.
\end{enumerate}
As regards the polynomial analogue, unlike the case of continuous linear functionals which always extend to the ambient space, polynomials defined on subspaces do not always extend to the larger space. For example, the scalar polynomial $P$ on $\ell_{2}$ given by  $P\big((x_{i})\big)=\sum_{i\ge 1}x_{i}^{2}$ cannot be extended to a polynomial on $\ell_{\infty}$. However, as in the case of linear functionals on subspaces of a Hilbert space where projections can be used to extend mappings from subspaces to the larger space, polynomials also admit extensions from subspaces of a Hilbert space to the whole space. Further, there are non-Hilbertian Banach spaces which admit extension of polynomials from subspaces to the larger space. Indeed, Maurey's extension theorem guarantees the extendability of bilinear forms (and therefore of scalar polynomials of degree $2$) for type $2$ Banach spaces. This raises the question of the description of Banach spaces $X$ such that each polynomial on each subspace of $X$ admit an extension to a polynomial of the same degree on $X$. Whereas it remains unknown whether there are non-Hilbertain Banach spaces witnessing the aforementioned extension property, it turns out that the existence and boundedness of the linear `extension' operator between appropriate spaces of polynomials (of a fixed degree $n$) involving the given Banach space $X$ does indeed characterise $X$ as a Hilbert space. In what follows, the symbol $P^n(Z)$ denotes the space of polynomials of degree at most $n$ on a Banach space $Z$.

\begin{theorem} \cite{{FK},{FP}}
For a Banach space $X$, TFAE
\end{theorem}
\begin{enumerate}
\item[(i)] $X$ is (isomorphic to ) a Hilbert space.
\item[(ii)] For each subspace $Y$ of $X$, there exists a bounded linear map $\psi: P^{n}(Y) \to P^{n}(X)$ such that $\psi(g)|_{Y}= g$ for each $g\in P^{n}(Y)$.
\item[(iii)]  For each subspace $Y$ of $X$, there exists a bounded linear map $\psi: Y^{*}\to X^{*}$ such that $\psi(g)|_{Y}= g$ for each $g\in Y^{*}$.
\end{enumerate}

\section{Main results}
We begin with the Lipschitz analogue of Theorem 1.1 and ask if it is possible to choose for each subspace $Y$ of $X$ a bounded linear `extension' map $\psi: \textnormal{Lip}(Y) \to \textnormal{Lip}(X)$, i.e. the map $\psi$ verifies the condition: $\psi(g)|_{Y}= g$ for each $g\in \textnormal{Lip}(Y)$. Here $\textnormal{Lip}(M)$ stands for the space of all Lipschitz functions $f: M \to \mathbb R$, defined on a metric space $M$, whereas we shall mostly  consider  the (sub) space $Lip_0(M)$ of $Lip(M)$ consisting of functions vanishing at a distinguished point, say $\theta \in M$. It is easily checked that under pointwise operations, $\textnormal{Lip}(M)$ is a Banach space when equipped with the norm:
$$\big\|f\big\|=\sup_{x \ne y}\dfrac{\big|f(x)- f(y)\big|}{d(x,y)}.$$

For a comprehensive and the state of art account on Banach space-valued Lipschitz mappings, their extensions in the setting of metric/Banach spaces and of bilipschitz embeddings of metric spaces into nice Banach spaces/Hilbert spaces, one may look into the recent monograph \cite{CN}. Another useful source is \cite{wea} providing a thorough description of the theory underlying the space $Lip_0(M)$ as an algebra and its cannonical predual, the free space $\Im(X)$ of $X$.

The Lipschitz equivalent of Hahn-Banach theorem is an old result of McShane (see \cite{H} Theorem 6.2) which says that a real-valued Lipschitz function can always be extended from an arbitrary subset of a metric space to the whole space. We shall see below that for subspaces $Y$ of a Banach space $X$, the existence of a bounded linear `extension' map $\psi: \textnormal{Lip}_0(Y) \to \textnormal{Lip}_0(X)$ imposes rather severe restrictions on $Y$. We shall need the following definition.
\begin{definition}  A subspace $Y$ of a Banach space $X$ is said to be locally complemented if there exists $c > 0$ such that for each finite dimensional subspace $M$ of $X$, there exists a continuous linear map $f: M \to Y$ with $\big\|f\big\| \le c$ and $f(x)=x$ for all $x \in M \cap Y$.
\end{definition}

Of course a complemented subspace is trivially locally complemented. However, to find examples of locally complemented subspaces which are not complemented, we recall the well-known fact that $c_{0}$ is a Lipschitz retract of $\ell_{\infty}$ (see \cite{BL}, Example 1.5). In particular, $c_{0}$ is locally complemented in $\ell_{\infty}$. (This will be seen to follow from Corollary $2.5$ below). Again, as is well known, $c_{0}$ is not complemented in $\ell_{\infty}$.

The following proposition provides a list of useful necessary and sufficient conditions for a subspace to be locally complemented.

\begin{proposition}\cite{{FK},{K1}}
~\\
Given a subspace $Z$ of a Banach space $X$, the following conditions are equivalent:
\begin{enumerate}
\item $Z$ is locally complemented.
\item $Z^{\perp}$, the annihilator of $Z$ is complemented in $X^{*}$.
\item There exists a bounded linear extension map $G:Z^{*} \longrightarrow X^{*}$.
\end{enumerate}

\end{proposition}

It turns out that there are examples of Banach spaces $X$ and subspaces $Z$ admitting no bounded linear extension maps from $\textnormal{Lip}_0(Z)$ into $\textnormal{Lip}_0(X)$. The (corollary to the) following theorem guarantees the existence of such subspaces inside each Banach space which is non-isomorphic to a Hilbert space.
\begin{theorem}\label{t3}
 Let $X$ be a Banach space and $Z$ a subspace of $X$ such that there exists a bounded linear map `extension' map $F: \textnormal{Lip}_0(Z) \to \textnormal{Lip}_0(X)$. Then $Z$ is locally complemented in $X$.
\end{theorem}
\begin{proof}
 We begin by using $F$ to define another bounded linear map $G: \textnormal{Lip}_0(Z) \to \textnormal{Lip}_0(X)$ which takes values in $X^*$ and when restricted to $Z^*$ yields a bounded linear extension map $G:Z^*\longrightarrow X^*$. The map $G$ is constructed using an old idea of Pelczynski (see \cite{L}) involving the existence of a Banach limit- hereinafter
to be denoted by $\int . ~ dx$ - which is a continuous linear functional on $\ell_\infty(X)$ satisfying the following conditions:
\begin{enumerate}
\item[(a)]  $\big\|\int . . dx\big\|=1$.
\item[(b)]  $\int1.dx=1$.
\item[(c)]  $\int f(x+x^\prime)dx=\int f(x)dx,\; \forall \, f \in \ell_\infty(X)$ and $x^\prime \in X$.
\end{enumerate}
The new map $G: \textnormal{Lip}_0(Z) \to \textnormal{Lip}_0(X)$ is now defined by the formula
$$G(f)(z)=\int\left\{\int\Big[(Ff)(x+y+z)-(Ff)(x+y)\Big]dy\right\}dx,\, f \in \textnormal{Lip}(Z),\, z \in X.$$
The linearity and boundedness of $F$ together with (a) and the Lipschitz property of $f$ yield that this $G: \textnormal{Lip}_0(Z) \to \textnormal{Lip}_0(X)$ is a well defined bounded linear map such that for $z_1, z_2 \in X$, we have, by virtue of (c),
\begin{align*}
G(f)(z_1+z_2)&=\int\left\{\int\Big[(Ff)(x+y+z_1+z_2)-(Ff)(x+y)\Big]dy\right\}dx\\
&=\int\left\{\int\Big[(Ff)(x+y+z_1+z_2)-(Ff)(x+y+z_2)\Big]dy\right\}dx\\
&\qquad\qquad+\displaystyle\int\left\{\int\Big[(Ff)(x+y+z_2)-(Ff)(x+y)\Big]dy\right\}dx\\
&=\int\left\{\int\Big[(Ff)(x+y+z_1)-(Ff)(x+y)\Big]dy\right\}dx\\\
&\qquad\qquad+\displaystyle\int\left\{\int\Big[(Ff)(x+y+z_2)-(Ff)(x+y)\Big]dy\right\}dx\\
&=G(f)(z_1)+G(f)(z_2).
\end{align*}
In other words, $G$ take values in $X^*$ and, therefore, restricting $G$ to $Z^*$ gives a bounded linear map $G: Z^* \to X^*$. To show that  $G$ is also an 'extension map' - as indeed $F$ is - choose $g \in  Z^*$ and $z \in Z.$ Writing
$$H(x,z)=\int[(Fg)(x+y+z)-(Fg)(x+y)]dy,~(x\in X,~z\in Z),$$
we see that
$$ G(g)(z)=\int H(x,z)dx.$$
Using (c), we have
\begin{align*}
H(x,z)&=\int\Big[(Fg)(x+y+z)-(Fg)(x+y)\Big]dy\\
&=\int\Big[(Fg)(x+y+z)-(Fg)(y+z)\Big]dy+\displaystyle\int\Big[(Fg)(y+z)-(Fg)(x+y)\Big]dy\\
&=\int\Big[(Fg)(x+y)-(Fg)(y)\Big]dy+\displaystyle\int\Big[(Fg)(y+z)-(Fg)(x+y)\Big]dy\\
&=\int\Big[(Fg)(y+z)-(Fg)(y)\Big]dy.
\end{align*}
Finally using (b) we have
\begin{align*}
G(g)(z)&=\int H(x,z)dx\\ &=\int\left\{\int\Big[(Fg)(y+z)-(Fg)(y)\Big]dy\right\}dx \\&=\int\left\{\int\Big[g(y+z)-g(y)\Big]dy\right\}dx\\
&=\int\left\{\int g(z)dy\right\}dx\\
&=g(z).
\end{align*}
In other words, $G: Z^* \to X^*$ is an ‘extension’ operator. By Proposition $2.2$, $Z$ is locally complemented in $X$ and this completes the proof of the theorem.
\end{proof}

\begin{corollary}\label{c4}
 For a Banach space $X$, TFAE:
\begin{enumerate}
\item[(i)]  $X$ is a Hilbert space.
\item[(ii)] For each (closed) subspace $Z$ of $X$, there exists a bounded linear `extension' map $\psi: \textnormal{Lip}_0(Z) \to \textnormal{Lip}_0(X)$.
\item[(iii)] For each (closed) subspace $Z$ of $X$, there exists a bounded linear `extension' $\psi: Z^* \to X^*$.
\item[(iv)]  Each (closed) subspace of $X$ is locally complemented.
\end{enumerate}
\end{corollary}

The proof of (i) $\Rightarrow$ (ii) is straightforward whereas (ii) $\Rightarrow$ (iii) follows from Theorem $2.3$. The equivalence (iii) $\Leftrightarrow$ (iv) follows from Proposition $2.2$. Finally, the equivalence (iii) $\Leftrightarrow$ (i) has been noted in Theorem $1.1$.
\begin{corollary}\label{c5}
 A Lipschitz retract is always locally complemented.
\end{corollary}

As a first important consequence of the above corollary, the following Lipschitz analogue of an isomorphic characterisation of Hilbert spaces follows directly from it.

\begin{corollary}\label{c6}
 For a Banach space $X$, TFAE:
\begin{enumerate}
\item[(i)] $X$ is a Hilbert space.
\item[(ii)] For each closed subspace $Z$ of $X$ and each Banach space $Y$, each Lipschitz map on $Z$ and taking values in $Y$ can be extended to a Lipschitz map on $X$ into $Y$.
\end{enumerate}
\end{corollary}
Indeed, condition (ii) applied to each closed subspace $Z$ of $X$ with $Y=Z$ and the identity map on $Y$ yields a bounded linear extension map $\psi: \textnormal{Lip}_0(Z)\to \textnormal{Lip}_0(X)$ which, by virtue of Corollary $2.4$, yields that $X$ is a Hilbert space.
\begin{remark}
 (i)  The converse of Corollary $2.5$ regarding the existence of a locally complemented subspace which is not a Lipschitz retract of the ambient space is intimately connected with the famous open problem whether a Banach space is always a Lipschitz retract of its bidual (see \cite{BL}, pp.$183$). In an important work, Kalton \cite{K2} solves the problem in the negative in the nonseparable case. However for separable Banach spaces, the problem remains
open. On the other hand, it is not difficult to see that every Banach space is locally complemented in its bidual which, however, may not be complemented as is testified by the space $c_{0}$ as a (closed) subspace of (its bidual) $\ell_{\infty}$.

(ii) A Banach space $X$ is said to be $\lambda$-injective ($\lambda\ge 1$) if each $X$-valued continuous linear mapping $f$ on any subspace of a given Banach space can be extended to a continuous linear  map $g$ on the whole space such that $\|g\|\le \lambda \|f\|$. It can be shown that a $\lambda$-injective Banach space is a $\lambda$-absolute Lipschitz retract $(\lambda-ALR)$, i.e., it is a $\lambda$-Lipschitz retract of every metric space containing it. Further, an application of Theorem $1.1$ yields that for a Banach space $X$ with the latter property, its bidual $X^{**}$ is $\lambda$-injective. Indeed, the given condition combined with the canonical inclusion of $X$ into $C(B_{X^*}, w^*)$ produces a continuous linear extension operator $ \textnormal{Lip}_0(X)\to \textnormal{Lip}_0(C(B_{X^*}, w^*))$. By Theorem $1.1$, we get a continuous linear extension map $\psi: X^*\longrightarrow (C(B_{X^*},w^*))^*$ with $\parallel\psi\parallel \leq \lambda$. Taking conjugates yields a continuous linear map $\psi^*: C(B_{X^*}, w^*)^{**}\to X^{**}$ which when restricted to  $X^{**}$ yields $X^{**}$ as a $\lambda$-Lipschitz retract of $C(B_{X^*}, w^*)^{**}$. Finally, since the latter space is always $1$-injective (see \cite{AK}, Proposition 4.3.8), it follows that $X^{**}$ is $\lambda$-injective. This renders Corollary 4.11 ((iii)$\Leftrightarrow$ (iv)) of \cite{JS} incorrect where it is wrongly claimed that $X$ is $\lambda$-injective. Here the authors have based their argument on an incorrect use of corollary 3.3 \cite{FK}.In fact, as can be seen from this argument, for a $\lambda$-absolute Lipschitz retract to be $\lambda$-injective, it has to be assumed that the space in question is also an ultrasummand, that is a Banach space which is complemented in its bidual.That the assumption of being an ultrasummand cannot be dispensed with follows by considering the Banach space $c_{0}$. Further, Kalton’s example \cite{K2} referred to in Remark $2.7 (i)$ above shows that a Banach space $X$ with an injective bidual may not be an $ALR$.

The question regarding the exact description of Banach spaces $X$ resulting from (ii) in the above corollary with $Z=\ell_{2}$ is an important open problem belonging to this circle of ideas. Maurey's extension theorem as quoted in the introduction says that the stated condition holds for linear maps on type 2 Banach spaces and taking values in a Hilbert space. To
the best of our knowledge, it seems to be unknown if the Lipschitz analogue of this
result actually holds; more precisely, if Lipschitz mapping defined on subspaces of type 2 Banach spaces and taking values in a Hilbert space admit a Lipschitz extension on the whole space. On the other hand, and as a partial converse to this statement, it is possible to show that under the stated condition, the space in question has type $p$ for $p<2$. This follows upon combining certain results of Johnson and Lindenstrauss with the following fact ($\ast$) which can be derived from the given condition:\vspace*{0.5em}

($\ast$) Given $X$, a Banach space of type 2, there exists a constant $c>0$ such that for each finite subset $S$ of $X$ and each Lipschitz function $f:S\longrightarrow\ell_{2}$, there exists a Lipschitz function $g:X\longrightarrow\ell_{2}$ extending $f$ such that $Lip(g)\leq cLip(f)$.\vspace*{0.5em}

Complete details of the proof of $(*)$ shall appear elsewhere.
Further, it is conjectured that such spaces are actually type 2 Banach spaces. Surprisingly, the linear analogue of the conjecture also remains open. On the other hand, using contractions in place of Lipschitz maps in this assertion, results in $X$ being (isometrically) a Hilbert space (see \cite{BL}, Theorem $2.11$ and the remark following it).

(iii) In view of its relevance to the theme of a recent  work by the author \cite{S2} devoted to certain aspects of nonlinear analysis in Banach spaces, Theorem 2.3 (and some of its corollaries) also appear in the said volume with proofs of some of these results having since yielded to slight modifications/improvements as included in the present paper.
\end{remark}

The next corollary deals with the Lipschitz analogue of a familiar result involving the description of Banach spaces $X$ such that every closed convex subset of $X$ is a non-expansive retract of $X$. That this property holds for Hilbert spaces follows from the easily-checked observation that the `nearest point projection’ $P_{C}$ corresponding to a given closed subset $C$ of $X$ is indeed a non-expansive (1-Lipschitz) map which is the identity map on $C$. Here $P_{C}$ is defined to be the map given by:
$$P_{C}(x)=\inf \Big\{ \big\|x-y\big\|: y\in C\Big\}, \quad x\in X .$$
The fact that the stated property holds precisely for Hilbert spaces is a famous theorem due originally to Reich \cite{R}. The next result provides the isomorphic analogue of the latter statement involving Lipschitz retracts. The proof follows exactly on similar lines as in the previous corollary.
\begin{corollary}\label{c7}
 A Banach space has the property that each closed and convex subset of $X$ is a Lipschitz retract of $X$ if and only if $ X$ is isomorphic to a Hilbert space.
\end{corollary}
\begin{comment}
An interesting generalisation of the above corollary can be provided by considering the more general notion of a $K$-random projection which subsumes the notion of a $K$-Lipschitz retraction. L. Ambrosio and D. Puglisi \cite{AP} show that the existence of a $K$-random projection leads to the existence of a bounded linear extension operator between appropriate spaces of Lipschitz functions
on closed subsets of (pointed) metric spaces. In particular, there exists a bounded, linear extension operators $\textnormal{Lip}_0(A)\longrightarrow\textnormal{Lip}_0(M)$ involving doubling metric spaces $M$ and closed subspaces $A$ of $M$.
\end{comment}

We now comment on the question involving the existence of an  'extension' operator $\psi: \textnormal{Lip}_0(A) \to \textnormal{Lip}_0(M)$, where A is now chosen to be an arbitrary subset of the metric space M. If such is the case, we say that the metric space M has the ‘simultaneous Lipschitz extension property’ (SLEP). An answer to this question turns out to be highly nontrivial, especially in the case of $X$ being a non-Banach metric space.  We note that the existence of an `extension'
operator entails (i) extendability of Lipschitz mappings from arbitrary subsets to the entire space (ii) a suitable choice of the extension which may be effected in a bounded linear manner. In the case of Banach spaces, the existence of such a choice forces the space to be finite dimensional, as we shall see shortly.  However, restricting to only part (i) of the extension procedure with mappings consisting of contractions and taking values inside  $\ell_2$ results in $X$ being a Hilbert space, as was noted in the Comments following Remarks $2.7$.  This is a well known converse of Kirszbraun's theorem \cite{Ki}.

Beyond the class of Banach spaces, this question has been a subject of intensive research spanning different areas of mathematics including, in particular, geometry and group theory. There are interesting examples of situations where such extensions always exist which include, for example,  $\mathbb R^n$ (with respect to any norm), the Heisenberg group, doubling metric spaces, metric trees  and certain classes of Riemannian manifolds of bounded geometry. On the other hand, making use of some deep but well known  facts from the local theory of Banach spaces, it is possible to show that the extension procedure for arbitrary subsets breaks down for all infinite dimensional Banach spaces, including of course, infinite dimensional Hilbert spaces where the situation was seen to be extremely pleasant from the viewpoint of the extension procedure involving linear functionals and Lipschitz mappings acting on linear subspaces. We provide below a sketch of the proof of this assertion which is a simplification of the proof already available in the literature (\cite{BB2}).

\begin{theorem}\label{t8}
 A Banach space $X$ is finite dimensional if and only if $X$ has (SLEP).
\end{theorem}
\begin{proof}[Sketch of the proof]
Let us introduce certain numerical parameters that will be used to quantify the existence of the extension operator mentioned above. To this end, for each subset $A$ of a metric space $M$, denote by $Ext(A, M)$ the set of `extension' operators $\psi: \textnormal{Lip}_0(A) \to \textnormal{Lip}_0(M)$, i.e.,
$\psi(g)|_{Y}= g$ for each $g\in \textnormal{Lip}_0(A)$. We set
$$\lambda(A, M)=\sup\Big\{\inf \big\|T\big\|: T \in Ext(A, M), \, A \subset M\Big\}$$
and define the `global Lipschitz extension constant' of $M$ to be the number
$$\lambda (M)=\sup\big\{\lambda(A, M), \, A \subset M\big\}.$$
We claim that $\lambda(X)=\infty$ for each infinite dimensional Banach space $X$. In particular, this would imply that each infinite dimensional Banach space admits a subset $A$ without the `simultaneous extension property'. We shall need the following tools to prove our statement.
\begin{enumerate}
\item[(a)] Dvoretzky's spherical sections theorem: Each infinite dimensional Banach space $X$ contains for each $n$, an $n$-dimensional subspace $X_n$ which is 2-isometric with $l_2^n$.
\item[(b)] There exists $c > 0$ (independent of $n$) such that $\lambda(l_2^n)\ge c n^{1/8}$.
\item[(c)] If $r:B\to A$ is a $c$-Lipschitz retraction of $B$ onto $A$, then $\lambda(A)\le c\lambda(B)$.
\end{enumerate}
Whereas (a) is already well known as a highly nontrivial fact from the local theory of Banach spaces (See \cite{BL}, Theorem $12.10$), the proof of (b) hinges on another deep result involving an alternative description of $\lambda(M)$ given by the following formula  (See \cite{BB2}):
$$\lambda(M)=\sup \big\{v(M,Z): Z \in FD\big\}.$$
Here $FD$ denotes the class of all finite dimensional Banach spaces and $v(M,Z)$ is given by
$$v(M,Z)=\sup \big\{v(A, M,Z): A \subset M\big\}.$$
where $v(A, M,Z)$ is defined to be the infimum of all constants $c > 0$ such that each $g \in \textnormal{Lip}_0(A, Z)$ admits an extension to $f \in \textnormal{Lip}_0(M, Z)$ such that
$$\big\|f\big\|_{\textnormal{Lip}(M, Z)}\le c \big\|g\big\|_{\textnormal{Lip}(A, Z)}.$$
Finally, considering that the proof of (c) does not seem to be explicitly available in the existing literature, we include a proof of it as follows. Let  $S\subset A$ and choose $T\subset B$   with $r(T)=S$. Now given a Lipschitz map $f:S\to \mathbb R$, we see that  $(f \circ r)|_{T}$ is a Lipschitz function on $T$. Fix $\varepsilon >0$ and choose an extension operator $E:\textnormal{Lip}_0(T)\to \textnormal{Lip}_0(B)$ with $\|E\|\le \big(\lambda(T,B)+\varepsilon\big)$. The assignment $E_{1}(f)=E\big ( (f \circ r)|_{T}\big)|_{A}$, then defines a continuous linear map
$$E_{1}: \textnormal{Lip}_0(S)\to \textnormal{Lip}_0(A).$$
From the inequality:
$$\big\|(f \circ r)|_{T}\big\|_{\textnormal{Lip}(T)}\le\|r\|_{\textnormal{Lip}(B,A)} \big\|f\big\|_{\textnormal{Lip}(S)},$$
we conclude that
$$\big\|E_{1}\big\|\le \|r\|_{\textnormal{Lip}(B,A)}\|E\|.$$
This gives:  $$\lambda(S,A)\le \|E_1\|\le \|r\|_{\textnormal{Lip}(B,A)}\|E\|\le \|r\|_{\textnormal{Lip}(B,A)} \big(\lambda(T,B)+\varepsilon\big).$$
Taking supremum on both sides and letting $\varepsilon\to 0$ yields the desired inequality.

Now to prove the desired assertion, for each $n> 0$, choose $X_n$ according to (a). Thus there exists a linear isomorphism
$\phi_n: \ell_2^n \to X_n$ such that $\|\phi_n\|\|\phi^{-1}_n\|< 2$. Fix $\varepsilon > 0$ and let $A \subset \ell_2^n$. There exists $ E \in \text{Ext}\big(\phi_n(A), X_n \big)$ such that $\|E\|<\lambda\big(\phi_n(A), X_n\big)+\varepsilon$. Consider the mappings $\Psi_n: \textnormal{Lip}(X_n) \to \textnormal{Lip}(\ell_2^n)$ and $L_n: \textnormal{Lip}(A) \to \textnormal{Lip}(\phi(A))$ given by
$$\Psi_n(f)=f\circ \phi_n\quad\text{and}\quad L_n(g)= g \circ \phi^{-1}_n,~ f \in \textnormal{Lip}(X_n), g \in \textnormal{Lip}(A).$$
Then $F$ given by $F=\Psi_n \circ E \circ L_n$ defines a linear map $F: \textnormal{Lip}(A) \to \textnormal{Lip}(\ell_2^n)$ such that
$$\big\|F\big\|\le \big\|\Psi_n\big\|\big\|E\big\|\big\|L_n\big\|\le \big\|\phi_n\big\|\big\|E\big\|\big\|\phi^{-1}_n\big\|\le 2\big\|E\big\|.$$
Using the above estimates, we get
$$\lambda \big(A, \ell_2^n\big)\le \big\|F\big\|\le 2\big\|E\big\|< 2 \big(\lambda (\phi_n(A), X_n)+\varepsilon \big)\le 2\big(\lambda ( X_n)+\varepsilon \big).$$
Letting $\varepsilon \to 0$ gives $\lambda(A, \ell_2^n)\le 2\lambda ( X_n)$.
Applying (c) to (the $2$-Lipschitz retract) $X_n$ of $X$, it follows that $\lambda(X_n)\le 2\lambda (X)$. Combined with the above estimate, this yields
$$\lambda(X) \ge \dfrac{1}{2} \lambda(X_n) \ge \dfrac{1}{4}\,\lambda \big(A, \ell_2^n\big).$$
Taking supremum over $A \subset \ell_2^n$ and using (b), we get
$$\lambda(X) \ge \dfrac{1}{4}\,\lambda \big(\ell_2^n \big)\ge \dfrac{c}{4}\, n^{1/8}.$$
Since $n$ is arbitrary, it follows that $\lambda(X)=\infty$. Finally, the assertion that $\lambda(X)<\infty$ for finite dimensional Banach spaces is a famous result due to Whitney (see \cite{St}, Chapter 6.2.3).
\end{proof}
\begin{remark}
(i) In contrast to Theorems $2.3$ and $2.10$ involving the existence /nonexistence of extension operators from spaces of Lipschitz functions on (closed) subspaces and arbitrary subsets of Banach spaces respectively, Kopecka \cite{Ko} shows that in the latter case involving arbitrary subsets $S$ of a Hilbert space $H$, there exist `extension' operators $\psi: \textnormal{Lip}(S,H)\to \textnormal{Lip}(B,H)$
preserving the Lipschitz norm and acting between spaces of vector-valued bounded Lipschitz functions which are continuous but not  linear. Here the spaces of Lipschitz functions are equipped with the sup-norm.

(ii) In light of Theorem $2.3$ and the corollaries following from it, a natural question arises regarding the extent of the validity of these conclusions under the assumption that for a subspace $Z$ of $X$,
the extension operator  $F: \textnormal{Lip}_0(Z) \longrightarrow \textnormal{Lip}_0(X)$ is assumed to be Lipschitz rather than linear. To the best of our understanding, the question does not seem to have been addressed in the literature so far.  Accordingly, we pose it as an open problem:

Question: Let $X$ be a Banach space. Does the existence of a Lipschitz extension operator $F: \textnormal{Lip}_0(Z) \longrightarrow \textnormal{Lip}_0(X)$ for each subspace $Z$ of $X$ yield that $X$ is a Hilbert space?

We note that for a subspace $Z$ of $X$ for which $\textnormal{Lip}_0(Z)$  is separable, the answer is in the affirmative. This follows from an important theorem of Godefroy and Kalton \cite{GK} to the effect that given Banach spaces $X$ and $Z$ with $Z$ separable and a bounded linear operator $T:X \longrightarrow Z$, the existence of a Lipschitz right inverse $S:Z\longrightarrow X$ of $T$ yields the existence of a bounded linear right inverse $L:Z \longrightarrow X$ of $T:T\circ L=Id_Z$. Indeed, consider the bounded linear map $\psi: \textnormal{Lip}_0(X) \longrightarrow \textnormal{Lip}_0(Z)$ given by: $\psi(f)=f_{\vert_Z}$.  Now the presumed existence of
$F: \textnormal{Lip}_0(Z) \longrightarrow \textnormal{Lip}_0(X)$  as a Lipschitz extension operator gives that $F$ is a Lipschitz right inverse of $\psi$. By the Godefroy-Kalton theorem just stated, we get a bounded linear right inverse $G$  of $\psi$, i.e., $G: \textnormal{Lip}_0(Z) \longrightarrow \textnormal{Lip}_0(X)$ is a bounded linear map such that $\psi \circ G=Id_{Lip_0(Z)}$. In other words, $G$ is a bounded linear extension operator and so we are in the situation of Theorem 2.3 to arrive at the desired conclusion.
\end{remark}

We conclude with a brief discussion of how, on the one hand, the parameter $\lambda(M)$ is related to the bounded approximation property (BAP) of the so called free Banach space $\Im (M)$ associated with the metric space $M$ and, on the other, how it is related to its vector-valued analogue $\lambda(M,Z)$, where Z is a Banach space. To this end, we briefly describe (without proofs) the Lipschitz free Banach space $\Im (M)$ over the metric space M which is defined as follows (\cite{GK}; see also \cite{{BB1},{wea}}).

It turns out that $\textnormal{Lip}_0(M)$ is a dual space, with its canonical predual being the closed linear span of the Dirac measures (of points of M) inside the Banach space $l_{\infty}(B)$ of bounded functions on $B$, the closed unit ball of $\textnormal{Lip}_0(M)$. The latter space, denoted by $\Im(M)$, shall be called the free space of $M$.
Further, given $f \in \textnormal{Lip}_{0}(M, X)$, there exists a linear operator $T_f : \textnormal{Lip}_{0}(M) \rightarrow X$ such that
\begin{enumerate}
\item[(a)] $T_{f}|_{S} =f$
\item[(b)] $T_{f+ g}|_{S} =T_{f}|_{S} + T_{g}|_{S}$
\item[(c)] $ \| T _{f} \| \leq \| f \|_{  \textnormal{Lip}_{0}(M, X)   } $.
\end{enumerate}
It turns out that a Banach space X has (BAP) if and only if $\Im(X)$ has it. Godefroy \cite{G} proves that $\Im(M)$ has (BAP) if $\lambda (M) < \infty$ and asks if the converse is also true.
It turns out that the answer would be negative if the following question could be settled in the affirmative.

QUESTION: Is it true that $\lambda(M)=\infty$, where $M$ is the closed unit ball of the dual of an infinite dimensional Banach space equipped with the weak*-topology?

We conjecture that the answer is in the affirmative. In \cite{BB1}, the authors go a step further to assert the possible validity of the above statement in the more general case of infinite dimensional compact metric spaces. (Open Problem (f) in Section 4).

In continuation of the theme of this paper involving the deep connection between nonlinear theory of Banach spaces and their geometric structure, we conclude by including the proof of a special case of a result which is conjectured to be true for all Banach spaces. The conjecture pertains to the numerical parameters $\lambda(M)$ introduced in the proof of Theorem 2.9 and its relation with $\lambda(M,Z)$ which is defined analogously for a pair $(M,Z)$ consisting of a metric space $M$ and a Banach space $Z$ by the formula
\begin{equation*}
\lambda(M,Z)= \sup \left\lbrace \lambda (A,M,Z) \ ; \ A \subset M \right\rbrace
\end{equation*}
where
\begin{equation*}
\lambda (A,M,Z)= \inf \left\lbrace \| T\| \  ; \ T \in \text{Ext}(A,M,Z) \right\rbrace
\end{equation*}

and $\text{Ext}(A,M,Z)$ denotes the set of `extension' operators $\psi : \text{Lip}_0(A,Z) \rightarrow \text{Lip}_0(M,Z)$.

It is conjectured that $\lambda(M)= \lambda(M,Z)$ for each Banach space $Z$. We show that the conjecture holds for a class of Banach spaces that includes the class of ultrasummands, i.e., Banach spaces which are complemented in their biduals.

\begin{theorem}
Let $M$ be a metric space and $Z$ a Banach space which is a Lipschitz retract of its bidual. Then $\lambda(M)= \lambda(M,Z)$.
\end{theorem}
\begin{proof}
Let $S$ be a subset of $M$. It is easy to see that $\lambda(M) \leq \lambda(M,Z)$. Indeed, if $\lambda(M,Z)= \infty$, there is nothing to prove. Thus, we assume that $\lambda(M,Z) < \infty$. Fix $x \in Z$ such that $\| x \| =1$ and choose $g \in Z^*$ such that $g(x)=1$ with $\|g\|=1$. Let $h : R \rightarrow Z$ be the linear isometry defined by $h(t)= t x$. It is easily seen that the map $E : \text{Lip}_0(S) \rightarrow \text{Lip}_0(S,Z)$ given by $E(f)= h \circ f $ defines a linear isometry. Now for a given $\epsilon>0$, there exists $T \in \text{Ext}(S,M,Z)$ such that $\|T \| < \lambda(S,M,Z) + \epsilon$. Then $\hat{T} = g \circ T \circ E : \text{Lip}(S) \rightarrow \text{Lip}(M)$ defines a linear map such that $\hat{T}(f)(s)=f(s)$, for each $s \in S$. In other words, $\hat{T} \in \text{Ext}(S,M)$ such that $\|\hat{T} \| \leq \|T \| < \lambda(S,M,Z) + \epsilon$. This gives $\lambda(S,M) < \lambda(S,M,Z) + \epsilon$. Since $S$ is arbitrary, this yields the desired inequality: $\lambda(M) \leq \lambda(M,Z)$.

For the reverse inequality, fix a subset $S$ of $M$ and assume that $\lambda(S,M)< \infty$. For each $\epsilon>0$, choose $T \in \text{Ext}(S,M)$ such that $\| T \| < \lambda(S,M)+ \epsilon$. Taking conjugates gives $T^* : \text{Lip}_0(M)^* = \Im(M)^{**} \rightarrow \text{Lip}_0(S)^*= \Im(S)^{**}$. Since $T$ is an extension operator, it follows that $T^*$ is continuous linear projection. Looking at $M$ as isometrically embedded into $\Im(M)^{**}$, it follows that $T^*$ restricted to $M$ acts as a Lipschitz map $T^{*}{\big|}_{M} : M \rightarrow \Im(S)^{**}$ such that $ \| T^{*}{\big|}_{M} \|_{\text{Lip}(M)}  \leq \| T \| $. To show that $\lambda(S,M,Z) \leq \lambda(S,M)$, let $f \in \text{Lip}_0(S,Z)$ and consider the associated linear map $T_{f} : \Im (S) \rightarrow Z$ which extends $f$ such that $\| T_{f} \| \leq \|  f     \|_{  \text{Lip}(S,Z)   }$. This gives $ T_{f}^{**} : \Im (S)^{**}  \rightarrow Z^{**}$ such that $T_{f}^{**}{\big|}_{S} = T_{f}{\big|}_{S} = T{\big|}_{S} =f $. By the conditions of the theorem, there exists a Lipschitz retraction $R : Z^{**} \rightarrow Z$. It is readily seen that the formula: $ \hat{T}(f) = R \circ T_{f}^{**} \circ T^{*}{\big|}_{M}$ defines a continuous linear operator $\hat{T} : \text{Lip}_0(S,Z) \rightarrow \text{Lip}_0(M,Z)$ with $\| \hat{T}(f) \| < \|f \|_{\text{Lip}(S,Z)} \left( \lambda(S,M)  + \epsilon \right)$.  Finally, since $\hat{T}(f){\big|}_{S} = f$, it follows that $\hat{T} \in \text{Ext}(S,M,Z)$ and that $\| \hat{T} \| < \lambda(S,M) + \epsilon$. This gives $\lambda(S,M,Z) \leq \lambda(S,M)$ and  yields the desired inequality as in the previous case.
\end{proof}

\section*{Acknowledgements}

Part of this work was carried out and completed during the author’s visit to the HRI, Allahabad during Dec.2018-Feb.2019. The author would like to thank Prof. Ratnakumar for his kind invitation and hospitality and the CSIR, New Delhi for providing support under its Emeritus Scientist-scheme 2015. The author would also like to thank the refrees for their remarks and suggestions which have led to an improvement in the presentation of the paper.


\bibliographystyle{amsplain}

\bibliographystyle{plain}
\bibliography{litteratur.bib}

\end{document}